\documentclass{article}
\usepackage{maa-monthly}
\usepackage{tikz}

\newcommand{\diag}{\operatorname{diag}}
\newcommand{\per}{\operatorname{per}}
\newcommand{\coeff}{\operatorname{coeff}}
\newcommand{\sym}{\operatorname{sym}}
\newcommand{\fix}{\operatorname{fix}}
\newcommand{\sgn}{\operatorname{sgn}}

\newcommand{\Sn}{{\text{S}_n}}

\theoremstyle{theorem}
\newtheorem{thm}{Theorem}
\newtheorem*{thm*}{Theorem}

\newtheorem{lem}[thm]{Lemma}

\theoremstyle{definition}
\newtheorem{defn}[thm]{Definition}

\newtheorem{ex}[thm]{Example}

\begin{document}

\title{Two Proofs of the Hamiltonian Cycle Identity}
\markright{Hamiltonian Cycle Identity}
\author{Hamilton Sawczuk and Edinah Gnang\thanks{Department of Applied Mathematics and Statistics, Johns Hopkins University, Baltimore, MD 21218}}  

\maketitle

\begin{abstract}
The Hamiltonian cycle polynomial can be evaluated to count the number of Hamiltonian cycles in a graph. It can also be viewed as a list of all spanning cycles of length $n$. We adopt the latter perspective and present a pair of original proofs for the Hamiltonian cycle identity which relates the Hamiltonian cycle polynomial to the important determinant and permanent polynomials. The first proof is a more accessible combinatorial argument. The second proof relies on viewing polynomials as both linear algebraic and combinatorial objects whose monomials form lists of graphs. Finally, a similar identity is derived for the Hamiltonian path polynomial.
\end{abstract}

We refer to spanning cycles and the graphs that contain them as Hamiltonian. Hamiltonicity has long fascinated graph theorists and complexity theorists. The exploration of conditions for Hamiltonicity has generated hundred of results, many of which have the following flavor: if a graph $G$ is ``sufficiently connected," for example has a large number of edges, then it is Hamiltonian \cite{gou91,gou03}. At the same time, the computational problem of testing Hamiltonicity plays a central role in complexity theory as one of Karp's 21 original NP-complete problems and a special case of the Traveling Salesman Problem \cite{kar72}. The study of Hamiltonian cycles led to the following definition.

\begin{defn}
    Let $[n]:=\{1,2,\hdots,n\}$. Given an $n\times n$ matrix $A$, the \textbf{Hamiltonian cycle polynomial} of $A$, sometimes denoted $\operatorname{ham}(A)$ or $HC_n(A)$, is
    \[
        P_{HC_n}(A):=\sum_{\sigma\in HC_n}\prod_{i\in[n]}a_{i,\sigma(i)},
    \]
    where $HC_n$ denotes the set of permutations on $[n]$ containing exactly one cycle.
\end{defn}

\begin{ex}
    Observe that
    \[
        P_{H_3}(A)=a_{12}a_{23}a_{31}+a_{13}a_{21}a_{32},
    \]
    and each monomial of $P_{HC_3}(A)$ corresponds to a Hamiltonian cycle below.

    \[
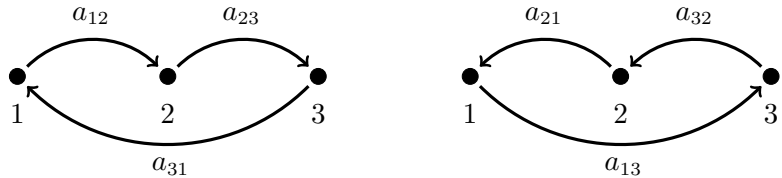
        \begin{array}{ccc}
             \begin{tikzpicture}
                \node (1) at (0,0) {};
                \node (2) at (2,0) {};
                \node (3) at (4,0) {};
                \draw[fill=black] (0,0) circle (3pt);
                \draw[fill=black] (2,0) circle (3pt);
                \draw[fill=black] (4,0) circle (3pt);
                \node at (0,-0.5) {$1$};
                \node at (2,-0.5) {$2$};
                \node at (4,-0.5) {$3$};
                \draw (1) edge[color=black,very thick,->,out=45,in=135,looseness=1] (2) node [above=23pt,right=17pt,fill=white] {$a_{12}$};
                
                \draw (2) edge[color=black,very thick,->,out=45,in=135,looseness=1] (3) node [above=23pt,right=17pt,fill=white] {$a_{23}$};
                
                \draw (3) edge[color=black, very thick, ->, out=225,in=315,looseness=1] (1) node [below=34pt,left=45pt,fill=white] {$a_{31}$};
            \end{tikzpicture} & \hspace{25pt} &
            \begin{tikzpicture}
                \node (1) at (0,0) {};
                \node (2) at (2,0) {};
                \node (3) at (4,0) {};
                \draw[fill=black] (0,0) circle (3pt);
                \draw[fill=black] (2,0) circle (3pt);
                \draw[fill=black] (4,0) circle (3pt);
                \node at (0,-0.5) {$1$};
                \node at (2,-0.5) {$2$};
                \node at (4,-0.5) {$3$};
                \draw (1) edge[color=black,very thick,<-,out=45,in=135,looseness=1] (2) node [above=23pt,right=17pt,fill=white] {$a_{21}$};
                
                \draw (2) edge[color=black,very thick,<-,out=45,in=135,looseness=1] (3) node [above=23pt,right=17pt,fill=white] {$a_{32}$};
                
                \draw (3) edge[color=black, very thick, <-, out=225,in=315,looseness=1] (1) node [below=34pt,left=45pt,fill=white] {$a_{13}$};
            \end{tikzpicture}
        \end{array}
    \]
    Thus we can think of $P_{HC_n}(A)$ as a list of Hamiltonian cycles on $[n]$. Alternatively, evaluating $P_{HC_n}(A)$ at the adjacency matrix of a graph $G$ returns the number of Hamiltonian cycles in $G$. Thus formulas expressing $P_{HC_n}$ are algorithms or arithmetic circuits for the counting variant of the Hamiltonian cycle problem. One such formula is the the Hamiltonian cycle identity, which we introduce after the following definition.
\end{ex}

\begin{defn}
    Given an $n\times n$ matrix $A$, the \textbf{determinant} of $A$ is
    \[  
        \det(A)=\sum_{\sigma\in\Sn}\sgn(\sigma)\prod_{i\in[n]}a_{i,\sigma(i)},
    \]
    and the \textbf{permanent} of $A$ is
    \[
        \per(A)=\sum_{\sigma\in\Sn}\prod_{i\in[n]}a_{i,\sigma(i)},
    \]
    where $\Sn$ is the group of permutations on $[n]$ and $\sgn$ is the sign function.
\end{defn}

\begin{thm*}
    The Hamiltonian cycle polynomial satisfies the identity
    \[
        P_{HC_n}(A)=\sum_{S\subseteq [n-1]}\det(-A_S)\per(A_{[n]\setminus S}),
    \]
    where $A_S$ denotes the principle submatrix of $A$ indexed by $S$, and $\det(A_{\varnothing}):=1$.
\end{thm*}

This identity is known, appearing for example in \cite{kne20} and closely related to the identities of \cite{koh77, kar82, bax93}. The contribution of this work is a pair of original proofs. Both proofs regard various polynomials as listing families of graphs. The first is a more accessible combinatorial argument. The second derives the Hamiltonian cycle identity from Tutte's directed matrix tree theorem using the determinant sum lemma and partial derivatives to distinguish graphs based on their degree sequences. We conclude with a similar identity for the Hamiltonian path polynomial which we believe to be new.

Although the Hamiltonian cycle identity does not represent a ``small" arithmetic circuit, the authors are struck by its beauty and use of the determinant and permanent polynomials, both of which are of highest mathematical significance \cite{val79a,val79b}. We now illustrate the identity with an example.

\begin{ex}
    When $n=3$,
    \[
        \begin{matrix}
            P_{H_3}(A) & = & \det(-A_\varnothing)\per(A) \\
            & & +\det(-A_{\{1\}})\per(A_{\{2,3\}})+\det(-A_{\{2\}})\per(A_{\{1,3\}}) \\
            & & +\det(-A_{\{1,2\}})\per(A_{\{3\}}) \\ \\
            & = & \textcolor{magenta}{a_{11}a_{22}a_{33}}+\textcolor{violet}{a_{11}a_{23}a_{32}}+\textcolor{blue}{a_{12}a_{21}a_{33}} \\
            & & +a_{12}a_{23}a_{31}+a_{13}a_{21}a_{32}+\textcolor{teal}{a_{13}a_{22}a_{31}}\\
            & & +(-a_{11})(\textcolor{magenta}{a_{22}a_{33}}+\textcolor{violet}{a_{23}a_{32}})+(-a_{22})(\textcolor{magenta}{a_{11}a_{33}}+\textcolor{teal}{a_{13}a_{31}}) \\
            & & +(\textcolor{magenta}{a_{11}a_{22}}-\textcolor{blue}{a_{12}a_{21}})a_{33} \\ \\
            & = & a_{12}a_{23}a_{31}+a_{13}a_{21}a_{32}.
        \end{matrix}
    \]
    Intuitively, $\per(A)$ lists all permutations on $[n]$, and all permutations with more than one cycle appear in $P_{HC_n}(A)$ the same number of times with positive and negative sign, canceling. We now introduce some vocabulary that will be helpful for both proofs of the theorem.
\end{ex}

\begin{defn}
    Given a function $f:[n]\rightarrow[n]$, the graph of $f$, denoted $G_f$, is a graph on $[n]$ with edge set $\{(i,f(i)):i\in[n]\}$. We say a graph $G$ is \textbf{functional} if it is the graph of some function $f$. Note that $G$ is \textbf{functional} if and only if $G$ has out-degree sequence $(1,1,\hdots,1)$.
\end{defn}

\begin{defn}
    Given a function $f:[n]\rightarrow[n]$, the \textbf{edge monomial} of the graph $G_f$ is
    \[
        M_{G_f}=\prod_{i\in[n]}a_{i,f(i)}.
    \]
    Observe that the monoids of functions on $[n]$, functional graphs on $[n]$, and edge monomials of degree $n$ are all isomorphic; we refer to objects from these sets interchangeably.
\end{defn}

\begin{defn}
    Given a family $\mathcal{F}$ of functions $f:[n]\rightarrow[n]$, the \textbf{symbolic listing} of $\mathcal{F}$ is
    \[
        P_{\mathcal{F}}=\sum_{f\in\mathcal{F}}M_{G_f}=\sum_{f\in\mathcal{F}}\prod_{i\in[n]}a_{i,f(i)}.
    \]
    According to this definition, $P_{HC_n}$ is the symbolic listing of Hamiltonian cycles.
\end{defn}

\begin{ex}
    Notice that according to the definitions we just introduced,
    \[
        \per(A)=\sum_{\sigma\in\Sn}\prod_{i\in[n]}a_{i,\sigma(i)}
    \]
    is the symbolic listing of permutations on $[n]$, also referred to as cycle covers. In the case $n=2$ we have
    \[
        \per(A)=a_{11}a_{22}+a_{12}a_{21},
    \]
    and each monomial of $\per(A)$ is the edge monomial of a graph below.
    \[
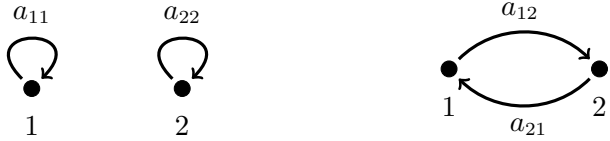
    \begin{array}{ccc}
         \begin{tikzpicture}
            \node (0) at (0,0) {};
            \node (1) at (2,0) {};
            \draw[fill=black] (0,0) circle (3pt);
            \draw[fill=black] (2,0) circle (3pt);
            \node at (0,-0.5) {$1$};
            \node at (2,-0.5) {$2$};
            \draw (0) edge[very thick,->,out=135,in=45,looseness=10] (0) node [above=21pt,fill=white] {{$a_{11}$}};
            \draw (1) edge[very thick,->,out=135,in=45,looseness=10] (1) node [above=21pt,fill=white] {{$a_{22}$}};
            \draw (0) edge[color=white,out=315,in=225,looseness=1] (1);
        \end{tikzpicture} & \hspace{50pt} &
        \begin{tikzpicture}
            \node (0) at (0,0) {};
            \node (1) at (2,0) {};
            \draw[fill=black] (0,0) circle (3pt);
            \draw[fill=black] (2,0) circle (3pt);
            \node at (0,-0.5) {$1$};
            \node at (2,-0.5) {$2$};
            \draw (0) edge[very thick,->,out=45,in=135,looseness=1] (1) node [above=22pt,right=16pt,fill=white] {{$a_{12}$}};
            \draw (1) edge[very thick,->,out=225,in=315,looseness=1] (0) node [below=22pt,left=16pt,fill=white] {{$a_{21}$}};
        \end{tikzpicture}
    \end{array}
    \]
    Recall also that a function $f:[n]\rightarrow[n]$ is a permutation if and only if it is a bijection. Further, $f$ is bijective if and only if every vertex of its graph $G_f$ has in-degree one, i.e. $G_f$ is the spanning union of directed cycles, i.e. $G_f$ has in-degree sequence $(1,1,\hdots,1)$.
\end{ex}

\section{Combinatorial Proof}\label{s:cproof}

\begin{thm*}
    The Hamiltonian cycle polynomial satisfies the identity
    \[
        P_{HC_n}(A)=\sum_{S\subseteq [n-1]}\det(-A_S)\per(A_{[n]\setminus S}),
    \]
    where $A_S$ denotes the principle submatrix of $A$ indexed by $S$, and $\det(A_{\varnothing}):=1$.
\end{thm*}

\begin{proof}
    Observe that only edge monomials of permutations appear in the expanded form of $P_{HC_n}$. To see this, fix $S\subseteq[n-1]$ and consider the term
    \[
    P_S:=\det(-A_S)\per(A_{[n]\setminus S}).
    \]
    Recall that $\det(-A_S)$ lists permutations on $S$ and $\per(A_{[n]\setminus S})$ lists permutations on $[n]\setminus S$. Thus each term in their product $P_S$ is the edge monomial of the disjoint union of a permutation on $S$ and a permutation on $[n]\setminus S$, i.e. a permutation.

    So we fix a permutation $\sigma\in\Sn$ and express $\sigma$ as the disjoint union of cycles $C_1,C_2,\hdots,C_k$. Without loss of generality suppose vertex $n$ is in $C_k$. Denote the coefficients of the edge monomial $M_\sigma$ in the expanded forms of $P_{HC_n}$ and $P_S$ by $\coeff(M_\sigma)$ and $\coeff_S(M_\sigma)$ respectively, so
    \[
        \coeff(M_\sigma)=\sum_{S\subseteq[n-1]}\coeff_S(M_\sigma).
    \]

    Now observe that the edge monomial $M_\sigma$ appears in the expanded form of $P_S$ if and only if $\{C_1,C_2,\hdots,C_k\}$ is a refinement of the partition $\{S,[n]\setminus S\}$, i.e. all cycles of $\sigma$ lie entirely in $S$ or in $[n]\setminus S$. Further, when $S$ is fixed this happens in one way, so $M_\sigma$ appears exactly once, but its sign could be positive or negative, i.e. $\coeff_S(M_\sigma)=\pm1$.
    
    Again without loss of generality suppose
    \[
        V(C_{r_1})\cup\hdots\cup V(C_{r_t})=S;
    \]
    \[
        V(C_{r_{t+1}})\cup\hdots\cup V(C_{r_{k-1}})\cup V(C_{k})=[n]\setminus S,
    \]
    where $0\leq t<k$ since by assumption vertex $n$ is in $C_k$ and $n\not\in S$. Then $\coeff_S(M_\sigma)$ is equal to the sign of $M_\sigma$ in
    \[
        \det(-A_S)\per(A_{[n]\setminus S})=(-1)^{|S|}\det(A_S)\per(A_{[n]\setminus S}).
    \]
    Since all terms of $\per(A_{[n]\setminus S})$ are positive, this is the same as the sign of $M_\sigma$ in $(-1)^{|S|}\det(A_S)$.
    Therefore
    \[
        \coeff_S(M_\sigma)=(-1)^{|S|}\sgn(\sigma|_S)=(-1)^{|S|}(-1)^{|S|-t}=(-1)^t,
    \]
    where $\sigma|_{S}$ denotes the restriction of $\sigma$ to $S$, and the second equality comes from the fact that $\sgn(\sigma)=(-1)^{n-k}$ when $k$ is the number of cycles in $\sigma$. Finally we have
    \[
        \coeff(M_\sigma)=\sum_{S\subseteq[n-1]}\coeff_S(M_\sigma)=\sum_{t=0}^{k-1}\sum_{
        \substack{
            S\subseteq[n-1] \\
            S=V(C_{r_1})\cup\hdots\cup V(c_{r_t})
        }
        }
        (-1)^t
    \]
    \[
        =\sum_{t=0}^{k-1}{k-1\choose t}(-1)^t=
        \begin{cases}
            1 & k=1 \\
            (1-1)^{k-1} & k\geq 2
        \end{cases},
    \]
    i.e. $\coeff(M_\sigma)=1$ when $\sigma$ is a Hamiltonian cycle and $\coeff(M_\sigma)=0$ otherwise.
    
\end{proof}

\section{Symbolic Proof Background}\label{s:background}

\begin{thm}
    Recall \textbf{Tutte's directed matrix tree theorem} (TDMTT), which provides the following symbolic listing of rooted trees.
    \[
        P_{T_n}(A) = \sum_{i=1}^na_{i,i}\det\left(\diag(A\boldsymbol{1})-A\right)_{[n]\setminus\{i\}}
    \]
\end{thm}

\begin{ex}
    In lieu of a proof we examine the case $n=3$.
    \[
        \diag(A\boldsymbol{1})-A
        =
        \begin{pmatrix}
            a_{12}+a_{13}   & -a_{12}       & -a_{13}       \\
            -a_{21}         & a_{21}+a_{23} & -a_{23}       \\
            -a_{31}         & -a_{32}       & a_{31}+a_{32} \\
        \end{pmatrix}
    \]
    \[
        \begin{matrix}
            P_{T_3}(A) & = & a_{11}[(a_{21}+a_{23})(a_{31}+a_{32})-(-a_{23})(-a_{32})] \\
            & & + a_{22}[(a_{12}+a_{13})(a_{31}+a_{32})-(-a_{13})(-a_{31})] \\
            
            & & + a_{33}[(a_{12}+a_{13})(a_{21}+a_{23})-(-a_{12})(-a_{21})] \\ \\
            
            & = & a_{11}[a_{21}a_{31}+a_{21}a_{32}+a_{23}a_{31}+\textcolor{red}{a_{23}a_{32}}-\textcolor{red}{a_{23}a_{32}}] \\
            & & + a_{22}[a_{12}a_{31}+a_{12}a_{32}+\textcolor{red}{a_{13}a_{31}}+a_{13}a_{32}-\textcolor{red}{a_{13}a_{31}}] \\
            & & + a_{33}[\textcolor{red}{a_{12}a_{21}}+a_{12}a_{23}+a_{13}a_{21}+a_{13}a_{23}-\textcolor{red}{a_{12}a_{21}}] \\ \\
            
            & = & \textcolor{magenta}{a_{11}a_{21}a_{31}}+\textcolor{teal}{a_{11}a_{21}a_{32}}+\textcolor{teal}{a_{11}a_{23}a_{31}} \\
            & & + \textcolor{teal}{a_{22}a_{12}a_{31}}+\textcolor{magenta}{a_{22}a_{12}a_{32}}+\textcolor{teal}{a_{22}a_{13}a_{32}} \\
            & & + \textcolor{teal}{a_{33}a_{12}a_{23}}+\textcolor{teal}{a_{33}a_{13}a_{21}}+\textcolor{magenta}{a_{33}a_{13}a_{23}}
        \end{matrix}
    \]
    Terms in the final expression are colored according to their conjugacy class membership. The graphs below depict these conjugacy classes.
    
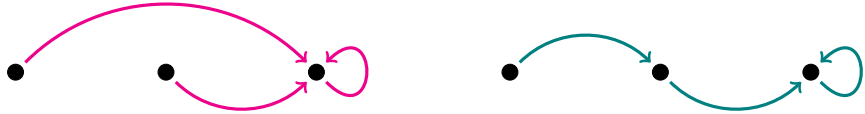
\[
        \begin{matrix}
            \begin{tikzpicture}
                \node (1) at (0,0) {};
                \node (2) at (2,0) {};
                \node (3) at (4,0) {};
                \draw[fill=black] (0,0) circle (3pt);
                \draw[fill=black] (2,0) circle (3pt);
                \draw[fill=black] (4,0) circle (3pt);
                \draw (3) edge[color=magenta,very thick,<-,out=45,in=-45,looseness=10] (3);
                \draw (2) edge[color=magenta,very thick,->,out=-45,in=225,looseness=1] (3);
                \draw (3) edge[color=magenta,very thick,<-,out=135,in=45, looseness=1] (1);
            \end{tikzpicture} & \hspace{25pt} & 
            \begin{tikzpicture}
                \node (1) at (0,0) {};
                \node (2) at (2,0) {};
                \node (3) at (4,0) {};
                \draw[fill=black] (0,0) circle (3pt);
                \draw[fill=black] (2,0) circle (3pt);
                \draw[fill=black] (4,0) circle (3pt);
                \draw (3) edge[color=teal,very thick,<-,out=45,in=-45,looseness=10] (3);
                \draw (2) edge[color=teal,very thick,->,out=-45,in=225,looseness=1] (3);
                \draw (1) edge[color=teal,very thick,->,out=45,in=135, looseness=1] (2);
            \end{tikzpicture}
        \end{matrix}
    \]
    Intuitively, each term $a_{i,i}\det(\diag(A\boldsymbol{1}-A))$ lists all functional graphs with a loop at vertex $i$, but any graphs with additional cycles appear the same number of times with positive and negative sign and thus cancel.
\end{ex}

We next present a formula for the determinant of a diagonal perturbation of a matrix. Although we are unable to find a published proof of this fact we believe the result is known.

\begin{lem}
    Given an $n\times n$ matrix $A$, the \textbf{determinant sum lemma} states
    \[
        \det(A+\diag(\boldsymbol{x}))=\sum_{S\subseteq[n]}\det(A_{S})\det(\diag(\boldsymbol{x})_{[n]\setminus S})
    \]
    where $A_S$ denotes the principle submatrix of $A$ indexed by $S$.
\end{lem}

\begin{proof}
    By the definition of the determinant,
    \[
        \det(A+\diag(\boldsymbol{x}))=\sum_{\sigma\in\Sn}\sgn(\sigma)\prod_{i\in[n]}(A+\diag(\boldsymbol{x}))_{i,\sigma(i)}.
    \]
    Observe that
    \[
        (A+\diag(\boldsymbol{x}))_{i,\sigma(i)}=
        \begin{cases}
            a_{i,\sigma(i)} & i\neq \sigma(i) \\
            a_{i,\sigma(i)}+x_{\sigma(i)} & i=\sigma(i),
        \end{cases}
    \]
    so letting $\fix(\sigma)$ denote the set of fixed points of $\sigma$,
    \[
        \det(A+\diag(\boldsymbol{x}))=\sum_{\sigma\in\Sn}\sgn(\sigma)\prod_{i\not\in \fix(\sigma)}a_{i,\sigma(i)}\prod_{i\in\fix(\sigma)}(a_{i,\sigma(i)}+x_{\sigma(i)}).
    \]
    Now we expand
    \[
        \prod_{i\in\fix(\sigma)}(a_{i,\sigma(i)}+x_{\sigma(i)})=\sum_{S\subseteq\fix(\sigma)}\prod_{i\in\fix(\sigma)\setminus S}a_{i,\sigma(i)}\prod_{i\in S}x_{\sigma(i)}
    \]
    to write
    \[
        \det(A+\diag(\boldsymbol{x}))=\sum_{\sigma\in\Sn}\sgn(\sigma)\prod_{i\not\in \fix(\sigma)}a_{i,\sigma(i)}\sum_{S\subseteq\fix(\sigma)}\prod_{i\in\fix(\sigma)\setminus S}a_{i,\sigma(i)}\prod_{i\in S}x_{\sigma(i)},
    \]
    and pulling out the inner sum,
    \[
        \det(A+\diag(\boldsymbol{x}))=\sum_{\sigma\in\Sn}\sum_{S\subseteq\fix(\sigma)}\sgn(\sigma)\prod_{i\not\in \fix(\sigma)}a_{i,\sigma(i)}\prod_{i\in\fix(\sigma)\setminus S}a_{i,\sigma(i)}\prod_{i\in S}x_{\sigma(i)}.
    \]
    Of course for $i\in\fix(\sigma)$, $i=\sigma(i)$, so substituting $x_i=x_{\sigma(i)}$ and combining the products of $a_{i,\sigma(i)}$,
    \[
        \det(A+\diag(\boldsymbol{x}))=\sum_{\sigma\in\Sn}\sum_{S\subseteq\fix(\sigma)}\sgn(\sigma)\prod_{i\in[n]\setminus S}a_{i,\sigma(i)}\prod_{i\in S}x_i.
    \]
    Next, switching the order of summation,
    \[
        \det(A+\diag(\boldsymbol{x}))=\sum_{S\subseteq[n]}\sum_{\substack{\sigma\in\Sn\\\fix(\sigma)\supseteq S}}\sgn(\sigma)\prod_{i\in[n]\setminus S}a_{i,\sigma(i)}\prod_{i\in S}x_i.
    \]
    The key here is to regard $\sigma\in\Sn$ with $S\subseteq\fix(\sigma)$ simply as a permutation on $[n]\setminus S$ since it fixes all elements of $S$. Denoting the set of permutations on $[n]\setminus S$ by $\sym([n]\setminus S)$,
    \[
        \det(A+\diag(\boldsymbol{x}))=\sum_{S\subseteq[n]}\left(\sum_{\sigma\in\sym([n]\setminus S)}\sgn(\sigma)\prod_{i\in[n]\setminus S}a_{i,\sigma(i)}\right)\prod_{i\in S}x_i,
    \]
    and recognizing the inner sum as the determinant of a submatrix of $A$ and the product on the right as the determinant of a submatrix of $\diag(\boldsymbol{x})$,
    \[
        \det(A+\diag(\boldsymbol{x}))=\sum_{S\subseteq[n]}\det(A_{[n]\setminus S})\prod_{i\in S}x_i
        =\sum_{S\subseteq[n]}\det(A_{[n]\setminus S})\det(\diag(\boldsymbol{x})_{S}).
    \]
    Finally, by symmetry
    \[
        \det(A+\diag(\boldsymbol{x}))=\sum_{S\subseteq[n]}\det(A_{S})\det(\diag(\boldsymbol{x})_{[n]\setminus S}).
    \]
\end{proof}

\begin{defn}
    For a multiset $S=\{s_i\}_{i=1}^k$ whose elements are drawn from $[n]$, let
    \[
        \boldsymbol{\partial}_S:=\frac{\partial}{\partial x_{s_1}}\frac{\partial}{\partial x_{s_2}}\hdots\frac{\partial}{\partial x_{s_k}},
    \]
    which we refer to as a differential operator of degree $|S|$. For example,
    \[
        \boldsymbol{\partial}_{[n]}=\frac{\partial}{\partial x_{1}}\frac{\partial}{\partial x_{2}}\hdots\frac{\partial}{\partial x_{n}}.
    \]
\end{defn}

\begin{lem}
    Recall the multivariable product rule, which states
    \[
        \boldsymbol{\partial}_{[n]}P(\boldsymbol{x})Q(\boldsymbol{x})
        =\sum_{S\subseteq [n]}\boldsymbol{\partial}_{S}P(\boldsymbol{x})\cdot\boldsymbol{\partial}_{[n]\setminus S}Q(\boldsymbol{x}).
    \]
\end{lem}
\begin{proof}
    When $n=1$ we have the product rule for partial derivatives,
    \[
        \frac{\partial}{\partial x_1}P(\boldsymbol{x})Q(\boldsymbol{x})=\left(\frac{\partial}{\partial x_1}P(\boldsymbol{x})\right)Q(\boldsymbol{x})+P(\boldsymbol{x})\left(\frac{\partial}{\partial x_1}Q(\boldsymbol{x})\right).
    \]
    Now suppose the claim holds for all $n\in[k-1]$. Then
    \[
        \boldsymbol{\partial}_{[k]}P(\boldsymbol{x})Q(\boldsymbol{x})=\boldsymbol{\partial}_{[k-1]}\left(\frac{\partial}{\partial x_k}P(\boldsymbol{x})Q(\boldsymbol{x})\right)
    \]
    \[
        =\boldsymbol{\partial}_{[k-1]}\left[\left(\frac{\partial}{\partial x_k}P(\boldsymbol{x})\right)Q(\boldsymbol{x})+P(\boldsymbol{x})\left(\frac{\partial}{\partial x_k}Q(\boldsymbol{x})\right)\right]
    \]
    \[
        =\boldsymbol{\partial}_{[k-1]}\left(\frac{\partial}{\partial x_k}P(\boldsymbol{x})\right)Q(\boldsymbol{x})+\boldsymbol{\partial}_{[k-1]}P(\boldsymbol{x})\left(\frac{\partial}{\partial x_k}Q(\boldsymbol{x})\right)
    \]
    \[
        =\sum_{S\subseteq [k-1]}\boldsymbol{\partial}_{S}\frac{\partial}{\partial x_k}P(\boldsymbol{x})\cdot\boldsymbol{\partial}_{[k-1]\setminus S}Q(\boldsymbol{x})
        +\sum_{T\subseteq [k-1]}\boldsymbol{\partial}_{T}P(\boldsymbol{x})\cdot\boldsymbol{\partial}_{[k-1]\setminus T}\frac{\partial}{\partial x_k}Q(\boldsymbol{x})
    \]
    \[
        =\sum_{S\subseteq [k]}\boldsymbol{\partial}_{S}P(\boldsymbol{x})\cdot\boldsymbol{\partial}_{[k]\setminus S}Q(\boldsymbol{x}),
    \]
    which completes induction.
\end{proof}

\section{Symbolic Proof}\label{s:sproof}

\begin{lem}\label{lem:PHn}
    The following expression is a symbolic listing of Hamiltonian cycles on $[n]$.
    \[
        P_{HC_n}(A)=\boldsymbol{\partial}_{[n]}\sum_{j\in[n]}a_{n,j}x_j\det(\diag(A\boldsymbol{x})-A\diag(\boldsymbol{x}))_{[n-1]}
    \]
\end{lem}
\begin{proof}
    TDMTT states
    \[
        P_{T_n}=\sum_{i\in[n]}a_{i,i}\det(\diag(A\boldsymbol{1})-A)_{[n]\setminus\{i\}}
    \]
    is a symbolic listing of rooted functional trees on $[n]$. Therefore
    \[
        \tilde P_{UC_n}(A)=\sum_{i\in[n]}\sum_{j\in[n]}a_{i,j}\det(\diag(A\boldsymbol{1})-A)_{[n]\setminus\{i\}}
    \]
    lists unicyclic graphs on $[n]$, since replacing the loop edge $(i,i)$ with any edge $(i,j)$ induces exactly one cycle in the graph of each monomial. Note however, that the coefficients in the above expression need not be one. Now consider
    \[
        \tilde P_{UC_n}(A\diag(\boldsymbol{x}))=\sum_{i\in[n]}\sum_{j\in[n]}a_{i,j}x_j\det(\diag(A\boldsymbol{x})-A\diag(\boldsymbol{x}))_{[n]\setminus\{i\}}
    \]
    which coincides with the image of $\tilde P_{UC_n}(A)$ after the change of variable $a_{i,j}\leftarrow a_{i,j}x_j$. In this expression, the variables of $\boldsymbol{x}$ record the in-degree sequence of each monomial. For example, the monomial $a_{12}a_{22}a_{32}$ has in-degree sequence $(1,2,0)$, and after the change of variable $a_{i,j}\leftarrow a_{i,j}x_j$ it appears with a factor of $x_1^1x_2^2x_3^0$. Thus if we take a partial derivative of $\tilde P_{UC_n}(A\diag(\boldsymbol{x}))$ with respect to $x_i$ for each $i\in[n]$, any monomial with an in-degree sequence containing a zero will vanish. Further, since $\tilde P_{UC_n}$ is homogeneous of degree $n$ in the variables of $\boldsymbol{x}$, its monomials have an in-degree sequence containing no zeros if and only if their degree sequence is $(1,1,\hdots,1)$. Then because functional graphs which are unicyclic and have in-degree sequence $(1,1,\hdots,1)$ are Hamiltonian cycles,
    \[
        \tilde P_{HC_n}(A)=\boldsymbol{\partial}_{[n]}\sum_{i\in[n]}\sum_{j\in[n]}a_{i,j}x_{j}\det(\diag(A\boldsymbol{x})-A\diag(\boldsymbol{x}))_{[n]\setminus\{i\}}
    \]
    lists Hamiltonian cycles on $[n]$. Finally, we modify the above expression to achieve coefficients of one. Observe that a monomial $M_{HC}$ of $\tilde P_{HC_n}$ is generated by a monomial $M_{T}$ of $P_{T_n}$ exactly when $M_{T}$ is the edge monomial of a functional path along the single spanning cycle of the graph of $M_{HC}$. Exactly one such $M_{T}$ is rooted at vertex $n$, so we restrict the above expression to $i=n$, yielding the desired expression for $P_{HC_n}(A)$. For example let $n=2$. Then the monomial $M_{HC}$ of the Hamiltonian cycle
    
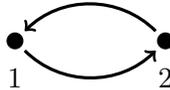
\[
        \begin{tikzpicture}
            \node (1) at (0,0) {};
            \node (2) at (2,0) {};
            \draw[fill=black] (0,0) circle (3pt);
            \draw[fill=black] (2,0) circle (3pt);
            \node at (0,-0.5) {$1$};
            \node at (2,-0.5) {$2$};
            \draw (1) edge[color=black,very thick,->,out=315,in=225,looseness=1] (2);
            
            \draw (2) edge[color=black,very thick,->,out=135,in=45,looseness=1] (1);
        \end{tikzpicture}
    \]
    is generated by the monomials $M_{T_1}$ and $M_{T_2}$ of the following functional paths, rooted at vertex $1$ and vertex $2$ respectively.
    \[
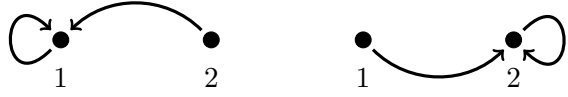
        \begin{array}{ccc}
            \begin{tikzpicture}
                \node (1) at (0,0) {};
                \node (2) at (2,0) {};
                \draw[fill=black] (0,0) circle (3pt);
                \draw[fill=black] (2,0) circle (3pt);
                \node at (0,-0.5) {$1$};
                \node at (2,-0.5) {$2$};
                \draw (1) edge[color=black,very thick,->,out=225,in=135,looseness=10] (1);
                
                \draw (2) edge[color=black,very thick,->,out=135,in=45,looseness=1] (1);
            \end{tikzpicture} & \hspace{25pt} & \begin{tikzpicture}
                \node (1) at (0,0) {};
                \node (2) at (2,0) {};
                \draw[fill=black] (0,0) circle (3pt);
                \draw[fill=black] (2,0) circle (3pt);
                \node at (0,-0.5) {$1$};
                \node at (2,-0.5) {$2$};
                \draw (1) edge[color=black,very thick,->,out=315,in=225,looseness=1] (2);
                
                \draw (2) edge[color=black,very thick,->,out=45,in=315,looseness=10] (2);
            \end{tikzpicture}
        \end{array}
    \]
    By restricting the expression for $\tilde P_{HC_2}(A)$ to $i=2$, we ensure $M_{HC}$ is only generated by $M_{T_2}$ and thus has coefficient one.

\end{proof}

\begin{thm*}
    The Hamiltonian cycle polynomial satisfies the identity
    \[
        P_{HC_n}(A)=\sum_{S\subseteq [n-1]}\det(-A_S)\per(A_{[n]\setminus S}),
    \]
    where $A_S$ denotes the principle submatrix of $A$ indexed by $S$, and $\det(A_{\emptyset}):=1$.
\end{thm*}

\begin{proof}
    By Lemma \ref{lem:PHn}, we have
    \[
        P_{HC_n}(A)=\boldsymbol{\partial}_{[n]}\sum_{j\in[n]}a_{n,j}x_j\cdot\det(\diag(A\boldsymbol{x})-A\diag(\boldsymbol{x}))_{[n-1]}
    \]
    is a symbolic listing of Hamiltonian cycles. Now observe
    \[
        \det(\diag(A\boldsymbol{x})-A\diag(\boldsymbol{x}))_{[n-1]}=\det(-A\diag(\boldsymbol{x}) + \diag(A\boldsymbol{x}))_{[n-1]},
    \]
    and by the determinant sum lemma
    \[
        \det(-A\diag(\boldsymbol{x}) + \diag(A\boldsymbol{x}))_{[n-1]}
    \]
    \[
        =\sum_{S\subseteq[n-1]}\det(-A\diag(\boldsymbol{x}))_S\det(\diag(A\boldsymbol{x}))_{[n-1]\setminus S}
    \]
    \[
        =\sum_{S\subseteq[n-1]}\det(-A\diag(\boldsymbol{x}))_S\per(\diag(A\boldsymbol{x}))_{[n-1]\setminus S}
    \]
    since the determinant and permanent of a diagonal matrix are the same. Plugging this in, we find $P_{HC_n}$ equals
    \[
        \boldsymbol{\partial}_{[n]}\sum_{j\in[n]}a_{n,j}x_j\sum_{S\subseteq[n-1]}\det(-A\diag(\boldsymbol{x}))_S\per(\diag(A\boldsymbol{x}))_{[n-1]\setminus S},
    \]
    which can be factored as
    \[
        \sum_{S\subseteq[n-1]}\boldsymbol{\partial}_{[n]}\det(-A\diag(\boldsymbol{x}))_S\sum_{j\in[n]}a_{n,j}x_j\per(\diag(A\boldsymbol{x}))_{[n-1]\setminus S}.
    \]
    Now fix $S\subseteq[n-1]$ and consider a single summand. Observe that the left and right hand factors above have degree $|S|$ and $n-|S|$ respectively in the variables of $\boldsymbol{x}$. By the multivariable product rule, the summand equals
    \[
        \sum_{T\subseteq[n]}\boldsymbol{\partial}_{T}\det(-A\diag(\boldsymbol{x}))_S
        \cdot\boldsymbol{\partial}_{[n]\setminus T}\sum_{j\in[n]}a_{n,j}x_j\per(\diag(A\boldsymbol{x}))_{[n-1]\setminus S}
    \]
    The left hand factor vanishes unless $T\subseteq S$. But if $|T|<|S|$ then the right hand factor vanishes since it has total degree $n-|S|$ but a degree $n-|T|>n-|S|$ differential operator. Thus the only non-vanishing term corresponds to $T=S$, and $P_{HC_n}(A)$ equals
    \[
        \sum_{S\subset[n-1]}\boldsymbol{\partial}_{S}\det(-A\diag(\boldsymbol{x}))_S \cdot \boldsymbol{\partial}_{[n]\setminus S}\sum_{j\in[n]}a_{n,j}x_j\per(\diag(A\boldsymbol{x}))_{[n-1]\setminus S}.
    \]
    Now
    \[
        \boldsymbol{\partial}_{S}\det(A\diag(\boldsymbol{x}))_S=\det(A)_S\cdot\boldsymbol{\partial}_S\det(\diag(\boldsymbol{x}))_S
    \]
    \[
        =\det(A)_S\cdot\boldsymbol{\partial}_{S}\prod_{i\in S}x_i=\det(A)_S
    \]
    which implies
    \[
        P_{HC_n}=\sum_{S\subset[n-1]}\det(-A)_S\cdot\boldsymbol{\partial}_{[n]\setminus S}\sum_{j\in[n]}a_{n,j}x_j\per(\diag(A\boldsymbol{x}))_{[n-1]\setminus S}.
    \]
    Finally, we examine the right hand factor
    \[
        \boldsymbol{\partial}_{[n]\setminus S}\sum_{j\in[n]}a_{n,j}x_j\per(\diag(A\boldsymbol{x}))_{[n-1]\setminus S}
    \]
    \[
        =\boldsymbol{\partial}_{[n]\setminus S}\sum_{j\in[n]}a_{n,j}x_j\prod_{i\in[n-1]\setminus S}\left(\sum_{k\in[n]}a_{i,k}x_{k}\right).
    \]
    Since the total degree is $n-|S|$, any non-vanishing term must have degree one in the variables of $\boldsymbol{x}$ indexed by $[n]\setminus S$ and degree zero in those indexed by $S$. Thus we can restrict the above sums to $j,k\in[n]\setminus S$ to write
    \[
        =\boldsymbol{\partial}_{[n]\setminus S}\sum_{j\in[n]\setminus S}a_{n,j}x_j\prod_{i\in[n-1]\setminus S}\left(\sum_{k\in[n]\setminus S}a_{i,k}x_{k}\right)
    \]
    \[
        =\boldsymbol{\partial}_{[n]\setminus S}\left(\prod_{i\in[n-1]\setminus S}\sum_{k\in[n]\setminus S}a_{i,k}x_{k}\right)\left(\sum_{j\in[n]\setminus S}a_{n,j}x_j\right)
    \]
    \[
        =\boldsymbol{\partial}_{[n]\setminus S}\prod_{i\in[n]\setminus S}\sum_{k\in[n]\setminus S}a_{i,k}x_{k}.
    \]
    Now we switch the order of product and summation, yielding
    \[
        =\boldsymbol{\partial}_{[n]\setminus S}\sum_{f:([n]\setminus S)\rightarrow([n]\setminus S)}\prod_{i\in[n]\setminus S}a_{i,f(i)}x_{f(i)},
    \]
    but as before, only monomials with in-degree sequence $(1,1,\hdots,1)$, i.e. permutations, will remain after the action of partial derivative operator. Thus we restrict the sum to find
    \[
        =\boldsymbol{\partial}_{[n]\setminus S}\sum_{\sigma\in\sym([n]\setminus S)}\left(\prod_{i\in[n]\setminus S}a_{i,\sigma(i)}x_{\sigma(i)}\right) =\boldsymbol{\partial}_{[n]\setminus S}\per(A\diag(\boldsymbol{x}))_{[n]\setminus S}.
    \]
    As with the determinant, 
    \[
        \boldsymbol{\partial}_{S}\per(A\diag(\boldsymbol{x}))_S=\per(A)_S\cdot\boldsymbol{\partial}_{S}\prod_{i\in S}x_i=\per(A)_S,
    \]
    so the right hand factor is equal to $\per(A_{[n]\setminus S})$, and therefore
    \[
        P_{HC_n}(A)=\sum_{S\subseteq[n-1]}\det(-A_S)\per(A_{[n]\setminus S})
    \]
    as desired.
\end{proof}

The argument above can be modified slightly to produce a similar identity for the Hamiltonian path polynomial.

\begin{thm*}
    The Hamiltonian path polynomial satisfies the identity
    \[
        P_{HP_n}(A)=\sum_{\{i,j\}\subseteq T\subseteq[n]}a_{j,j}\det(-A)_{[n]\setminus T}\per(A)_{T\setminus \{j\},T\setminus \{i\}}
    \]
\end{thm*}
\begin{proof}
    We begin with
    \[
        P_{HP_n}(A)=\sum_{i,j\in[n]}a_{j,j}\cdot\boldsymbol{\partial}_{[n]\setminus\{i\}}\det(\diag(A\boldsymbol{x})-A\diag(\boldsymbol{x}))_{[n]\setminus\{j\}}.
    \]
    To see this fix $i$ and $j$. By TDMTT,
    \[
        a_{j,j}x_j\det(\diag(A\boldsymbol{x})-A\diag(\boldsymbol{x}))_{[n]\setminus\{j\}}
    \]
    lists all functional trees rooted at vertex $j$ with their in-degree sequences recorded by the variables of $\boldsymbol{x}$. Then
    \[
        a_{j,j}\cdot\boldsymbol{\partial}_{[n]\setminus\{i\}}\det(\diag(A\boldsymbol{x})-A\diag(\boldsymbol{x}))_{[n]\setminus\{j\}}
    \]
    lists all functional $(i,j)$-paths since the partial derivative operator annihilates all terms except those with in-degree sequence $(0,1,\hdots,1,2)$ where vertex $i$ has in-degree zero and vertex $j$ has in-degree two. Summing over $i$ and $j$ we recover $P_{HP_n}$.

    By the determinant sum lemma,
    \[
        P_{HP_n}(A)=\sum_{i,j\in[n]}a_{j,j}\cdot\boldsymbol{\partial}_{[n]\setminus\{i\}}\det(\diag(A\boldsymbol{x})-A\diag(\boldsymbol{x}))_{[n]\setminus\{j\}}
    \]
    \[
        =\sum_{i,j\in[n]}a_{j,j}\cdot\boldsymbol{\partial}_{[n]\setminus\{i\}}\sum_{S\subseteq[n]\setminus\{j\}}\det(-A\diag(\boldsymbol{x}))_{S}\det(\diag(A\boldsymbol{x}))_{[n]\setminus\{j\}\setminus S},
    \]
    and by the multivariable chain rule,
    \[
        =\sum_{i,j\in[n]}a_{j,j}\sum_{S\subseteq[n]\setminus\{j\}}\sum_{T\subseteq[n]\setminus\{i\}}\boldsymbol{\partial}_{T}\det(-A\diag(\boldsymbol{x}))_{S}\cdot\boldsymbol{\partial}_{[n]\setminus\{i\}\setminus T}\det(\diag(A\boldsymbol{x}))_{[n]\setminus\{j\}\setminus S}.
    \]
    Again each term vanishes unless $S=T$ and the right hand factor
    \[
        \boldsymbol{\partial}_{[n]\setminus\{i\}\setminus S}\det(\diag(A\boldsymbol{x}))_{[n]\setminus\{j\}\setminus S}
    \]
    \[
        =\boldsymbol{\partial}_{[n]\setminus\{i\}\setminus S}\prod_{u\in[n]\setminus\{j\}\setminus S}\left(\sum_{v\in[n]}a_{u,v}x_v\right)
    \]
    \[
        =\boldsymbol{\partial}_{[n]\setminus\{i\}\setminus S}\prod_{u\in[n]\setminus\{j\}\setminus S}\left(\sum_{v\in[n]\setminus\{i\}\setminus S}a_{u,v}x_v\right)
    \]
    \[
        =\sum_{\substack{
            f:[n]\setminus S\setminus\{j\}\rightarrow[n]\setminus S\setminus\{i\} \\
            f\text{ bijective}
        }} \left(\prod_{v\in[n]\setminus\{j\}\setminus S}a_{v,f(v)}\right)
    \]
    \[
        =\per(A)_{[n]\setminus\{j\}\setminus S,[n]\setminus\{i\}\setminus S}.
    \]
    Thus we have
    \[
        P_{HP_n}(A)=\sum_{i,j\in[n]}\sum_{S\subseteq[n]\setminus\{i,j\}}a_{j,j}\det(-A\diag(\boldsymbol{x}))_{S}\per(A)_{[n]\setminus\{j\}\setminus S,[n]\setminus\{i\}\setminus S}.
    \]
    Now summing over $T=[n]\setminus S$ and combining the sums we see
    \[
        P_{HP_n}(A)=\sum_{\{i,j\}\subseteq T\subseteq[n]}a_{j,j}\det(-A)_{[n]\setminus T}\per(A)_{T\setminus \{j\},T\setminus \{i\}}.
    \]
\end{proof}

\begin{acknowledgment}{Acknowledgments.}
We thank Kailee Lin for assistance editing this manuscript.
\end{acknowledgment}

\bibliographystyle{vancouver}
\bibliography{main.bib}

\vfill\eject

\end{document}